\documentclass[reqno, 12pt]{amsart}
\usepackage{a4wide}
\usepackage{amscd}
\input xy
\xyoption{all}

\newtheorem{thm}{Theorem}[section]

\newtheorem{prop}[thm]{Proposition}
\newtheorem{cor}[thm]{Corollary}

\numberwithin{equation}{section}

\begin{document}

\title{Contact hypersurfaces in K\"{a}hler manifolds }
\author{\textsc{J\"{u}rgen Berndt} and  \textsc{Young Jin Suh}}

\address{King's College London\\ Department of Mathematics \\  London WC2R 2LS \\ United Kingdom}

\address{Kyungpook National University\\ Department of Mathematics \\  Taegu 702-701 \\ South Korea}

\date{}

\begin{abstract}
A contact hypersurface in a K\"{a}hler manifold is a real hypersurface for which the induced almost contact metric structure determines a contact structure. We carry out a systematic study of contact hypersurfaces in K\"{a}hler manifolds. We then apply these general results to obtain classifications of contact hypersurfaces with constant mean curvature in the complex quadric $Q^n = SO_{n+2}/SO_nSO_2$ and its noncompact dual space $Q^{n*} = SO^o_{n,2}/SO_nSO_2$ for $n \geq 3$.
\end{abstract}

\maketitle
\thispagestyle{empty}

\footnote[0]{2010 \textit{Mathematics Subject Classification}: Primary 53C40. Secondary 53C55, 53D15.\\
\textit{Key words}: Real hypersurfaces, contact hypersurfaces, K\"{a}hler manifolds, complex quadric}

\section{Introduction}
A contact manifold is a smooth $(2n-1)$-dimensional manifold $M$ together with a one-form $\eta$ satisfying $\eta \wedge (d\eta)^{n-1} \neq 0$, $n \geq 2$. The one-form $\eta$ on a contact manifold is called a contact form. The kernel of $\eta$ defines the so-called contact distribution ${\mathcal C}$ on $M$. Note that if $\eta$ is a contact form on a smooth manifold $M$, then $\rho\eta$ is also a contact form on $M$ for each smooth and everywhere nonzero function $\rho$ on $M$. The origin of contact geometry can be traced back to Hamiltonian mechanics and geometric optics. 

A standard example is a round sphere in an even-dimensional Euclidean space. Consider the sphere $S^{2n-1}(r)$ with radius $r \in {\mathbb R}_+$ in ${\mathbb C}^n$ and denote by $\langle \cdot , \cdot \rangle$ the inner product on ${\mathbb C}^n$ given by
$\langle z , w \rangle = {\rm Re} \sum_{\nu = 1}^n z_\nu \bar{w}_\nu$.
By defining
$\xi_z = -\frac{1}{r}iz$ for $z \in S^{2n-1}(r)$
we obtain a unit tangent vector field $\xi$ on $S^{2n-1}(r)$. We denote by $\eta$ the dual one-form given by
$\eta(X) = \langle X , \xi \rangle $
 and by $\omega$ the K\"{a}hler form on ${\mathbb C}^n$ given by
$\omega(X,Y) = \langle iX, Y \rangle$.
A straightforward calculation shows that
$d\eta(X,Y) = -\frac{2}{r}\omega(X,Y)$.
Since the K\"{a}hler form $\omega$ has rank $2(n-1)$ on the kernel of $\eta$ it follows that
$\eta \wedge (d\eta)^{n-1} \neq 0$.
Thus $S^{2n-1}(r)$ is a contact manifold with contact form $\eta$. 
This argument for the sphere motivates a natural generalization to K\"{a}hler manifolds. 

Let $(\bar{M},J,g)$ be a K\"{a}hler manifold of complex dimension $n$ and let $M$ be a connected oriented real hypersurface of $\bar{M}$. The K\"{a}hler structure on $\bar{M}$ induces an almost contact metric structure $(\phi,\xi,\eta,g)$ on $M$. The Riemannian metric on $M$ is the one induced from the Riemannian metric on $\bar{M}$, both denoted by $g$. The orientation on $M$ determines a unit normal vector field $N$ of $M$. The so-called Reeb vector field $\xi$ on $M$ is defined by $\xi = -JN$ and $\eta$ is the dual one form on $M$, that is, $\eta(X) = g(X,\xi)$. The tensor field $\phi$ on $M$ is defined by $\phi X  = JX - \eta(X)N$. Thus $\phi X$ is the tangential component of $JX$. The tensor field $\phi$ determines the fundamental $2$-form $\omega$ on $M$ by $\omega(X,Y) = g(\phi X,Y)$. $M$ is said to be a contact hypersurface if there exists an everywhere nonzero smooth function $\rho$ on $M$ such that
$d\eta = 2\rho \omega$. It is clear that if $d\eta = 2\rho \omega$ holds then $\eta \wedge (d\eta)^{n-1} \neq 0$, that is, every contact hypersurface in a K\"{a}hler manifold is a contact manifold.

Contact hypersurfaces in complex space forms of complex dimension $n \geq 3$ have been investigated and classified by Okumura \cite{O66} (for the complex Euclidean space ${\mathbb C}^n$ and the complex projective space ${\mathbb C}P^n$) and Vernon \cite{V87} (for the complex hyperbolic space ${\mathbb C}H^n$). In this paper we carry out a systematic study of contact hypersurfaces in K\"{a}hler manifolds. We will then apply our results to the complex quadric $Q^n = SO_{n+2}/SO_nSO_2$ and its noncompact dual space $Q^{n*} = SO^o_{n,2}/SO_nSO_2$. Here we consider $Q^n$ (resp.\ $Q^{n*}$) equipped with the K\"{a}hler structure for which it becomes a Hermitian symmetric space with maximal (resp.\ minimal) sectional curvature $4$ (resp.\ $-4$). The classification results for these two spaces are as follows.

\begin{thm}\label{classquadric}
Let $M$ be a connected orientable real hypersurface with constant mean curvature in the Hermitian symmetric space $Q^n = SO_{n+2}/SO_nSO_2$ and $n \geq 3$. Then $M$ is a contact hypersurface if and only if $M$ is congruent to an open part of the tube of radius $0 < r < \frac{\pi}{2\sqrt{2}}$ around the $n$-dimensional sphere $S^n$ which is embedded in $Q^n$ as a real form of $Q^n$.
\end{thm}

\begin{thm}\label{classquadricdual}
Let $M$ be a connected orientable real hypersurface with constant mean curvature in the Hermitian symmeric space $Q^{n*} = SO^o_{n,2}/SO_nSO_2$ and $n \geq 3$. Then $M$ is a contact hypersurface if and only if $M$ is congruent to an open part of one of the following contact hypersurfaces in $Q^{n*}$:
\begin{itemize}
\item[(i)] the tube of radius $r \in {\mathbb R}_+$ around the Hermitian symmetric space $Q^{(n-1)*}$ which is embedded in $Q^{n*}$ as a totally geodesic complex hypersurface;
\item[(ii)] a horosphere in $Q^{n*}$ whose center at infinity is the equivalence class of an ${\mathcal A}$-principal geodesic in $Q^{n*}$;
\item[(iii)] the tube of radius $r \in {\mathbb R}_+$ around the $n$-dimensional real hyperbolic space ${\mathbb R}H^n$ which is embedded in $Q^{n*}$ as a real form of $Q^{n*}$.
\end{itemize}
\end{thm}

The symbol ${\mathcal A}$ refers to a circle bundle of real structures on $Q^{n*}$ and the notion of ${\mathcal A}$-principal will be explained later. Every contact hypersurface in a K\"{a}hler manifold of constant holomorphic sectional curvature has constant mean curvature. Our results on contact hypersurfaces in K\"{a}hler manifolds suggest that it is natural to impose the condition of constant mean curvature in the more general setting.

In Section \ref{general} we will develop the general theory of contact hypersurfaces in K\"{a}hler manifolds. In Section \ref{application} we will apply these results to the complex quadric $Q^n$ and its noncompact dual $Q^{n*}$.

\smallskip {\sc Acknowledgments:} This work was supported by grant Proj.\ No.\ NRF-2011-220-1-C00002 from the National Research Foundation of Korea. The second author was supported by grant Proj. NRF-2012-R1A2A2A-01043023.

\section{Contact hypersurfaces in K\"{a}hler manifolds}\label{general}

Let $\bar{M}$ be a K\"{a}hler manifold of complex dimension $n$ and let $M$ be a connected oriented real hypersurface of $\bar{M}$. The hypersurface $M$ can be equipped with what is known as an almost contact metric structure $(\phi,\xi,\eta,g)$ which consists of
\begin{itemize}
\item[1.] a Riemannian metric $g$ on $M$ which is induced canonically from the K\"{a}hler metric (also denoted by $g$) on $\bar{M}$;
\item[2.] a tensor field $\phi$ on $M$ which is induced canonically from the complex structure $J$ on $\bar{M}$: for all vectors fields $X$ on $M$ the vector field $\phi X$ is obtained by projecting orthogonally the vector field $JX$ onto the tangent bundle $TM$;
\item[3.] a unit vector field $\xi$ on $M$ which is induced canonically from the orientation of $M$: if $N$ is the unit normal vector field on $M$ which determines the orientation of $M$ then $\xi = -JN$;
\item[4.] a one-form $\eta$ which is defined as the dual of the vector field $\xi$ with respect to the metric $g$, that is, $\eta(X) = g(X,\xi)$ for all $X \in TM$.
\end{itemize}
The vector field $\xi$ is also known as the Reeb vector field on $M$. The maximal complex subbundle ${\mathcal C}$ of the tangent bundle $TM$ of $M$ is equal to ${\rm ker}(\eta)$.

Let $S$ be the shape operator of $M$ defined by $SX = -\bar\nabla_X N$ ,
where $\bar\nabla$ denotes the Levi Civita covariant derivative on $\bar{M}$. 
Applying $J$ to both sides of this equation and using the fact that $\bar\nabla J = 0$ on a K\"{a}hler manfold implies
$ \phi S X = \nabla_X\xi $,
where $\nabla$ is the induced Levi Civita covariant derivative on $M$. 
Using again $\bar\nabla J = 0$ we get
\begin{eqnarray}\label{nablaphi}
(\nabla_X\phi)Y = \eta(Y)SX - g(SX,Y)\xi . 
\end{eqnarray}

Denote by $\omega$ the fundamental $2$-form on $M$ given by
$ \omega(X,Y) = g(\phi X,Y) $.

\begin{prop}\label{closed}
The fundamental $2$-form $\omega$ on a real hypersurface  in a K\"{a}hler manifold is closed, that is, $d\omega = 0$.
\end{prop}

\begin{proof}
By definition, we have
\begin{eqnarray*}
d\omega(X,Y,Z) & = & d(\omega(Y,Z))(X) + d(\omega(Z,X))(Y) + d( \omega(X,Y))(Z) \\
& & - \omega([X,Y],Z)  - \omega([Y,Z],X) - \omega([Z,X],Y) \\
& = & g(\nabla_X(\phi Y),Z) + g(\phi Y,\nabla_XZ) + g(\nabla_Y(\phi Z),X) + g(\phi Z,\nabla_YX) \\
& & + g(\nabla_Z(\phi X),Y) + g(\phi X,\nabla_ZY) \\
& & - g(\phi[X,Y],Z)  - g(\phi[Y,Z],X) - g(\phi[Z,X],Y) \\
& = & g((\nabla_X \phi) Y,Z) + g((\nabla_Y\phi) Z,X) + g((\nabla_Z\phi) X,Y).
\end{eqnarray*}
Inserting the expression for $\nabla \phi$ as in (\ref{nablaphi})  into the previous equation gives $d\omega = 0$.
\end{proof}

Motivated by the example $S^{2n-1}(r) \subset {\mathbb C}^n$ we say that $M$ is a contact hypersurface of $\bar{M}$ if there exists an everywhere nonzero smooth function $\rho$ on $M$ such that
$d\eta = 2\rho \omega $
holds on $M$ (Okumura \cite{O66}). It is clear that if this equation holds then $\eta \wedge (d\eta)^{n-1} \neq 0$, that is, every contact hypersurface of a K\"{a}hler manifold is a contact manifold.
Note that the equation $d\eta = 2\rho \omega$ means that
$ d\eta(X,Y) = 2\rho g(\phi X , Y) $
for all tangent vector fields $X,Y$ on $M$.
Using the definition for the exterior derivative we obtain
$d\eta(X,Y) = d(\eta(Y))(X) - d(\eta(X))(Y) - \eta([X,Y]) 
= g(Y,\nabla_X\xi) - g(X,\nabla_Y\xi) 
= g(Y,\phi S X) - g(X, \phi S Y) 
\break =  g((S\phi + \phi S)X,Y)$.
Thus we have proved:

\begin{prop}\label{char1}
Let $M$ be a connected orientable real hypersurface of a K\"{a}hler manifold $\bar{M}$. Then $M$ is a contact hypersurface if and only if there exists an everywhere nonzero smooth function $\rho$ on $M$ such that
\begin{equation}\label{char2}
S\phi + \phi S = 2\rho \phi.
\end{equation}
\end{prop}

A real hypersurface $M$ of a K\"{a}hler manifold is called a Hopf hypersurface if the flow of the Reeb vector field is geodesic, that is, if every integral curve of $\xi$ is a geodesic in $M$. This condition is equivalent to $0 = \nabla_\xi\xi = \phi S \xi$.
Since the kernel of $\phi$ is ${\mathbb R}\xi$ this is equivalent to
$S\xi = \alpha\xi $
with the smooth function $\alpha = g(S\xi,\xi)$. Since $\phi\xi = 0$, equation (\ref{char2}) implies $\phi S \xi = 0$ on a contact hypersurface, which shows that every contact hypersurface is a Hopf hypersurface. Let $X \in {\mathcal C}$ be a principal curvature vector of $M$ with corresponding principal curvature $\lambda$. Then equation (\ref{char2}) implies 
$S \phi X = (2\rho - \lambda)\phi X $,
that is, $\phi X$ is a principal curvature vector of $M$ with corresponding principal curvature $2\rho - \lambda$. Thus the mean curvature of $M$ can be calculated from $\alpha$ and $\rho$ from the equation
$ {\rm tr}(S) = \alpha + 2(n-1)\rho $.
We summarize this in:

\begin{prop}\label{Okumura}  Let $M$ be a contact hypersurface of a K\"{a}hler manifold $\bar{M}$. Then the following statements hold:
\begin{itemize}
\item[(i)] $M$ is a Hopf hypersurface, that is, $S\xi = \alpha \xi$.
\item[(ii)] If $X \in {\mathcal C}$ is a principal curvature vector of $M$ with corresponding principal curvature $\lambda$, then $JX = \phi X \in {\mathcal C}$ is a principal curvature vector of $M$ with corresponding principal curvature $2\rho - \lambda$.
\item[(iii)] The mean curvature of $M$ is given by
\begin{equation}\label{Aalpharho}
{\rm tr}(S) = \alpha + 2(n-1)\rho .
\end{equation}
\end{itemize}
\end{prop}

For $n = 2$ this gives a simple characterization of contact hypersurfaces:

\begin{prop}\label{dimtwo}
Let $M$ be a connected orientable real hypersurface of a $2$-dimensional K\"{a}hler manifold $\bar{M}^2$. Then $M$ is a contact hypersurface if and only if $M$ is a Hopf hypersurface and ${\rm tr}(S) \neq \alpha$ everywhere.
\end{prop}

\begin{proof}
The "only if" part has been proved in Proposition \ref{Okumura}. Assume that $M$ is a Hopf hypersurface. Then we have $S\xi = \alpha\xi$ and the maximal complex subbundle ${\mathcal C}$ of $TM$ is invariant under the shape operator $S$ of $M$. Let $X \in {\mathcal C}$ be a principal curvature vector with corresponding principal curvature $\lambda$. Since the rank of ${\mathcal C}$ is equal to $2$ the vector $JX = \phi X$ must be a principal curvature vector of $M$. Denote by $\mu$ the corresponding principal curvature. Then we have
$S\phi X + \phi S X = (\mu + \lambda)\phi X $
and 
$ S\phi (\phi X )+ \phi S (\phi X) = -(\lambda + \mu)X = (\lambda + \mu)\phi (\phi X)$.
This shows that  the equation
$ S\phi + \phi S = 2\rho \phi $
holds with $2\rho = (\lambda + \mu) = {\rm tr}(S) - \alpha$. It follows from Proposition \ref{char1} that $M$ is a contact hypersurface precisely if ${\rm tr}(S) \neq \alpha$.
\end{proof}

The previous result implies that there is a significant difference between the cases $n = 2$ and $n > 2$. For example, using Proposition \ref{dimtwo} we can construct many examples of locally inhomogeneous contact hypersurfaces in the complex projective plane ${\mathbb C}P^2$. In contrast, as was shown by Okumura (\cite{O66}), every contact hypersurface in the complex projective space ${\mathbb C}P^n$ of dimension $n > 2$ is an open part of a homogeneous hypersurface. Consider ${\mathbb C}P^2$ being endowed with the standard K\"{a}hler metric of constant holomorphic sectional curvature $4$. Let $C$ be a complex curve in ${\mathbb C}P^2$. Then, at least locally and for small radii, the tubes around $C$ are well-defined real hypersurfaces of ${\mathbb C}P^2$. All these real hypersurfaces are Hopf hypersurfaces with $\alpha = 2\cot(2r)$ where $r$ is the radius, and generically their mean curvature is different from $2\cot(2r)$.
For this reason we focus here on the case $n > 2$. 

\begin{prop}\label{rhoconstant}
Let $M$ be a connected real hypersurface of an $n$-dimensional K\"{a}hler manifold $\bar{M}^n$, $n > 2$, and assume that there exists an everywhere nonzero smooth function $\rho$ on $M$ such that $d\eta = 2\rho \omega$. Then $\rho$ is constant.
\end{prop}

\begin{proof}
Taking the exterior derivative of the equation  $d\eta = 2\rho \omega$ and using the fact that $\omega$ is closed gives $0 = d^2\eta = 2d\rho \wedge \omega$,
or equivalently,
\begin{equation}\label{cyclicdrho}
0 = d\rho(X)g(\phi Y,Z) + d\rho(Y)g(\phi Z,X) + d\rho(Z)g(\phi X,Y).
\end{equation}
For $X = \xi$, $Y \in {\mathcal C}$ with $||Y|| = 1$, and $Z = \phi Y$, this implies
$d\rho(\xi) = 0$.
Let $X \in {\mathcal C}$. Since $n > 2$ we can choose a unit vector $Y \in {\mathcal C}$ which is perpendicular to both $X$ and $\phi X$. Inserting $X,Y$ and $Z = \phi Y$ into equation (\ref{cyclicdrho}) gives
$d\rho(X) = 0$.
Altogether, since $M$ is connected, this implies that $\rho$ is constant.
\end{proof}

We denote by $\bar{R}$ the Riemannian curvature tensor of $\bar{M}$. For $p \in M$ and $Z \in T_p\bar{M}$ we denote by $Z_{\mathcal C}$ the orthogonal projection of $Z$ onto ${\mathcal C}$.

\begin{prop}\label{Asquared}
Let $M$ be a contact hypersurface of a K\"{a}hler manifold $\bar{M}$. Then we have
$2(S^2  - 2\rho S + \alpha\rho) X 
= (\bar{R}(JN,N)JX)_{\mathcal C}$
for all $X \in {\mathcal C}$.
In particular, for all $X \in {\mathcal C}$ with $SX = \lambda X$ we have
$
2(\lambda^2  - 2\rho \lambda + \alpha\rho )X = (\bar{R}(JN,N)J X)_{\mathcal C} $.
\end{prop}

\begin{proof}\label{Proof}
Since $M$ is a contact hypersurface, we know from Proposition \ref{Okumura} that $S\xi = \alpha \xi$.
Using the Codazzi equation and Proposition \ref{char1} we get for arbitrary tangent vector fields $X$ and $Y$ that
\begin{eqnarray*}
g(\bar{R}(X,Y)\xi,N) & = &   g((\nabla_XS)Y - (\nabla_YS)X, \xi) 
\ = \  g((\nabla_XS)\xi,Y) - g((\nabla_YS)\xi,X) \\
&=& d\alpha(X)\eta(Y) - d\alpha(Y)\eta(X) + \alpha g((S\phi + \phi
S)X,Y) - 2g(S\phi SX,Y) \\
&=& d\alpha(X)\eta(Y) - d\alpha(Y)\eta(X) + 2g((S^2  - 2\rho S + \alpha\rho)\phi X,Y).
\end{eqnarray*}
For $X = \xi$ this equation yields
$d\alpha(Y) = d\alpha(\xi)\eta(Y) + g(\bar{R}(Y,\xi)\xi ,N)$.
Since $\bar{M}$ is a K\"{a}hler manifold we have
$g(\bar{R}(Y,\xi)\xi ,N) = g(\bar{R}(JY,J\xi)\xi ,N) = g(\bar{R}(JY,N)\xi ,N) 
= g(J\bar{R}(\xi,N)N,Y)$,
and therefore
\begin{equation}\label{dalpha}
d\alpha(Y) = d\alpha(\xi)\eta(Y) + g(J\bar{R}(\xi,N)N,Y).
\end{equation}
Inserting this and the corresponding equation for $d\alpha(X)$ into the previous equation gives
$0 = 2g((S^2  - 2\rho S + \alpha\rho)\phi X,Y) - g(\bar{R}(X,Y)\xi,N) 
 - \eta(X)g(J\bar{R}(\xi,N)N,Y) + \break \eta(Y)g(J\bar{R}(\xi,N)N,X)$.
Choosing $X \in {\mathcal C}$, replacing $X$ by ${\phi}X$, and using some standard curvature identities then leads to 
the equation in Proposition \ref{Asquared}.
\end{proof}

\begin{prop}\label{alphaconstant}
Let $M$ be a contact hypersurface of a K\"{a}hler manifold $\bar{M}$. Then $\alpha$ is constant if and only if $JN$ is an eigenvector of the normal Jacobi operator $\bar{R}_N = \bar{R}(\cdot,N)N$ everywhere.
\end{prop}

\begin{proof}
First assume that $JN$ is an eigenvector of the normal Jacobi operator $\bar{R}_N = \bar{R}(\cdot,N)N$ everywhere. From equation (\ref{dalpha}) we get
$d\alpha(Y)  =  d\alpha(\xi)\eta(Y)$
for all $Y \in TM$.
Since ${\rm grad}^M \alpha =d\alpha(\xi)\xi$, we can compute the Hessian ${\rm hess}^M \alpha$ by
$
({\rm hess}^M \alpha)(X,Y) = g(\nabla_X{\rm grad}^M \alpha , Y)
= d(d\alpha(\xi))(X)\eta(Y) +d\alpha(\xi) g(\phi SX,Y)
$.
As ${\rm hess}^M \alpha$ is a symmetric bilinear form, the previous equation implies
$
0 = d\alpha(\xi)g((S\phi + \phi S)X,Y) = 2\rho d\alpha(\xi)g(\phi X,Y)
$
for all vector fields $X,Y$ on $M$ which are tangential to ${\mathcal C}$. Since $\rho$ is nonzero everywhere this implies $d\alpha(\xi) = 0$ and hence $\alpha$ is constant.

Conversely, assume that $\alpha$ is constant. From  (\ref{dalpha}) we get
$
0 = g(J\bar{R}(\xi,N)N,Y) = g(\bar{R}(JN,N)N,JY)
$
for all tangent vectors $Y$ of $M$. Since $J(TM) = {\mathcal C} \oplus {\mathbb R}N$ this implies that $\bar{R}(JN,N)N \in {\mathbb R}JN$ everywhere. 
\end{proof}

\begin{prop}\label{traceA2}
Let $M$ be a contact hypersurface of a K\"{a}hler manifold $\bar{M}$. Then we have
\[ ||S||^2 = {\rm tr}(S^2) = \alpha^2 + 2(n-1)\rho(2\rho - \alpha) - g(\overline{Ric}(N),N) + g(\bar{R}(JN,N)N,JN) \ , \]
where $\overline{Ric}$ is the Ricci tensor of $\bar{M}$.
\end{prop}

\begin{proof}
We choose a local orthonormal frame field of $\bar{M}$ along $M$ of the form 
$E_1,E_2 = J E_1,\ldots,E_{2n-3},E_{2n-2} = J E_{2n-3},E_{2n-1} = \xi,E_{2n} = J\xi = N$.
Since $\bar{M}$ is K\"{a}hler, its Ricci tensor $\overline{Ric}$ can by calculated by 
\[ \overline{Ric}(X) = \sum_{\nu=1}^{n} \bar{R}(E_{2\nu-1},JE_{2\nu-1})JX 
= \sum_{\nu=1}^{n} \bar{R}(E_{2\nu},JE_{2\nu})JX\]
along $M$ (see e.g.\ Proposition 4.57 in \cite{B06}). Using Propositions \ref{Asquared} and  \ref{Okumura} we get
\begin{eqnarray*}
{\rm tr}(S^2) & = & \alpha^2 + \sum_{\nu = 1}^{2n-2}g(S^2E_\nu,E_\nu) \\
& = & \alpha^2 + \sum_{\nu = 1}^{2n-2}\left(2\rho g(SE_\nu,E_\nu) - \alpha\rho g(E_\nu,E_\nu) + \frac{1}{2}g(\bar{R}(JN,N)JE_\nu,E_\nu)\right) \\
& = & \alpha^2 + 2\rho({\rm tr}(S) - \alpha) - 2(n-1)\alpha\rho -\frac{1}{2}\sum_{\nu = 1}^{2n-2}g(\bar{R}(E_\nu,JE_\nu)JN,N) \\
& = & \alpha^2 + 4(n-1)\rho^2 - 2(n-1)\alpha\rho - \left(g(\overline{Ric}(N),N) - g(\bar{R}(N,JN)JN,N)\right) \\
& = & \alpha^2 + 2(n-1)\rho(2\rho - \alpha) - g(\overline{Ric}(N),N) + g(\bar{R}(JN,N)N,JN)\ ,
\end{eqnarray*}
which proves the assertion.
\end{proof}

\begin{prop}\label{kminustraceA}
Let $M$ be a contact hypersurface of a K\"{a}hler manifold $\bar{M}^n$, $n > 2$. Then we have
$d({\rm tr}(S))(X)  =   g(\bar{R}(JN,N)N,JX)$
for all $X \in {\mathcal C}$. 
\end{prop}

\begin{proof}\label{Proof}
Since $M$ is a contact hypersurface, the equation $\phi S+S \phi=2\rho \phi $ holds, and since $n > 2$, the function $\rho$ is constant. Differentiating this equation leads to
\begin{equation*}
(\nabla_Y\phi)SX+\phi(\nabla_YS)X+(\nabla_YS)\phi X+S(\nabla_Y\phi)X= 2\rho (\nabla_Y\phi)X.
\end{equation*}
Using (\ref{nablaphi}) this implies
\[ 0 = 
\eta(X)(S^2Y+(\alpha-2\rho) SY)-g(S^2X+(\alpha -2\rho) SX,Y)\xi+\phi(\nabla_YS)X+(\nabla_YS)\phi X .
\]
We choose a local orthonormal frame field of $M$ of the form 
$ E_1,E_2 = J E_1,\ldots,E_{2n-3}$, $E_{2n-2} = J E_{2n-3},E_{2n-1} = \xi$.
Contracting the previous equation with respect to this frame field and using the formulas for ${\rm tr}(S)$ and ${\rm tr}(S^2)$ according to Propositions \ref{Okumura} and \ref{traceA2}, respectively, then gives
\[
0 = 
\left( g(\overline{Ric}(N),N) - g(\bar{R}(JN,N)N,JN) \right) \xi 
+\sum_{\nu=1}^{2n-1}\phi(\nabla_{E_{\nu}}S)E_\nu
+\sum_{\nu=1}^{2n-2}(\nabla_{E_{\nu}}S)\phi E_\nu.
\]
Using the  Codazzi equation 
$g((\nabla_XS)Y,Z) - g((\nabla_YS)X,Z) = g(\bar{R}(X,Y)Z,N)$
we get
\begin{eqnarray*}
& & \sum_{\nu=1}^{2n-1}g(\phi(\nabla_{E_{\nu}}S)E_\nu,X) =
- \sum_{\nu=1}^{2n-1}g((\nabla_{E_{\nu}}S)E_\nu,\phi X)
= - \sum_{\nu=1}^{2n-1}g((\nabla_{E_{\nu}}S)\phi X,E_\nu) \\
&  & = - \sum_{\nu=1}^{2n-1}g((\nabla_{\phi X}S)E_\nu,E_\nu) 
-  \sum_{\nu=1}^{2n-1}g(\bar{R}(E_\nu,\phi X)E_\nu,N) \\
& & = - {\rm tr}(\nabla_{\phi X} S)  + g(\phi X, \overline{Ric}(N)) 
 = - d({\rm tr}(S))(\phi X) + g(\phi X, \overline{Ric}(N)).
\end{eqnarray*}
Since $\nabla_XS$ is symmetric and $\phi$ is skewsymmetric, we get
$\sum_{\nu=1}^{2n-2}g((\nabla_XS)E_\nu,\phi E_\nu) = 0$,
and using again the Codazzi equation we obtain
\begin{eqnarray*}
& & \sum_{\nu=1}^{2n-2}g((\nabla_{E_{\nu}}S)\phi E_\nu,X) =
\sum_{\nu=1}^{2n-2}g((\nabla_{E_{\nu}}S)X,\phi E_\nu) \ = \ \sum_{\nu=1}^{2n-2}g(\bar{R}(E_\nu,X)\phi E_\nu, N) \\
& & = \sum_{\nu=1}^{2n-2}g(\bar{R}(E_\nu,X)J E_\nu, N) \ =\  \sum_{\nu=1}^{2n-2}g(J\bar{R}(E_\nu,X)E_\nu, N) 
= \sum_{\nu=1}^{2n-2}g(\bar{R}(E_\nu,X)E_\nu, \xi) \\
& & =  - g(\bar{R}(JN,N)N,X) - g(X,\overline{Ric}(\xi)).
\end{eqnarray*}
Altogether this now implies
\begin{eqnarray*}
0 & = & 
\left( g(\overline{Ric}(N),N) - g(\bar{R}(JN,N)N,JN) \right) \eta(X) \\
& & - d({\rm tr}(S))(\phi X) + g(\phi X, \overline{Ric}(N)) 
 - g(\bar{R}(JN,N)N,X) - g(X,\overline{Ric}(\xi)) .
\end{eqnarray*}
Using the fact that the Ricci tensor $\overline{Ric}$ and the complex structure $J$ of a K\"{a}hler manifold commute one can easily see that 
$\eta(X)g(\overline{Ric}(N),N) + g(\phi X, \overline{Ric}(N))  - g(X,\overline{Ric}(\xi)) = 0 $,
and therefore
$d({\rm tr}(S))(\phi X) = - g(\bar{R}(JN,N)N,X_{\mathcal C}) = - g((\bar{R}(JN,N)N)_{\mathcal C},X)$.
Replacing $X$ by $\phi X = JX$ for $X \in {\mathcal C}$ then leads to the assertion.
\end{proof}

\begin{prop}\label{kminustraceAconstant}
Let $M$ be a contact hypersurface of a K\"{a}hler manifold $\bar{M}^n$, $n > 2$. Then $M$ has constant mean curvature if and only if $JN$ is an eigenvector of the normal Jacobi operator $\bar{R}_N = \bar{R}(\cdot,N)N$ everywhere.
\end{prop}

\begin{proof}
We first assume that $JN$ is an eigenvector of the normal Jacobi operator $\bar{R}_N$ everywhere.
We put $f = {\rm tr}(S)$ and $\sigma = df(\xi)$. From Proposition \ref{kminustraceA} we see that
$ df = \sigma \eta $
and hence
$0 = d\sigma \wedge \eta + \sigma d\eta = d\sigma \wedge \eta + 2\rho \sigma \omega$.
In other words, we have
$0 = d\sigma(X)\eta(Y) - d\sigma(Y)\eta(X) + 2\rho\sigma g(\phi X, Y)$.
Replacing $Y$ by $\phi Y$ leads to
$0 = - d\sigma(\phi Y)\eta(X) + 2\rho\sigma g(\phi X, \phi Y)$.
By contracting this equation we obtain
$0 = 4(n-1)\rho\sigma$.
As $\rho$ is nonzero everywhere this implies $\sigma = 0$ and hence $df = 0$, which means that $f =  {\rm tr}(S)$ is constant. 

Conversely, assume that the mean curvature of $M$ is constant. From Proposition \ref{kminustraceA} we get
$0 = g(\bar{R}(JN,N)N,JX) $
for all $X \in {\mathcal C}$. This implies that $\bar{R}(JN,N)N$ is perpendicular to ${\mathcal C}$ everywhere.  Since $g(\bar{R}(JN,N)N,N) = 0$ we conclude that $\bar{R}(JN,N)N \in {\mathbb R}JN$, that is, $JN$ is an eigenvector of $\bar{R}_N$ everywhere.
\end{proof}

From Proposition \ref{alphaconstant} and Proposition \ref{kminustraceAconstant} we immediately get:

\begin{prop}\label{allandrhoconstant}
Let $M$ be a contact hypersurface of a K\"{a}hler  manifold $\bar{M}^n$, $n > 2$. Then the following statements are equivalent:
\begin{itemize}
\item[(i)] $\alpha$ is constant;
\item[(ii)] $M$ has constant mean curvature;
\item[(iii)] $JN$ is an eigenvector of the normal Jacobi operator $\bar{R}_N = \bar{R}(\cdot,N)N$ everywhere.
\end{itemize}
\end{prop}

The Riemannian universal covering of an $n$-dimensional K\"{a}hler manifold $\bar{M}$ with constant holomorphic sectional curvature $c$ is either the complex projective space ${\mathbb C}P^n$, the complex Euclidean space ${\mathbb C}^n$ or the complex hyperbolic space ${\mathbb C}H^n$ equipped with the standard K\"{a}hler metric of constant holomorphic sectional curvature $c > 0$, $c = 0$ and $c < 0$, respectively. The Riemannian curvature tensor $\bar{R}$ of an $n$-dimensional K\"{a}hler manifold $\bar{M}$ with constant holomorphic sectional curvature $c$ is given by
\[
\bar{R}(X,Y)Z = \frac{c}{4}\big(g(Y,Z)X - g(X,Z)Y + g(JY,Z)JX - g(JX,Z)JY - 2g(JX,Y)JZ\big).
\]
This implies $\bar{R}_N JN = \bar{R}(JN,N)N = cJN$,
and hence $JN$ is an eigenvector of the  Jacobi operator $\bar{R}_N = \bar{R}(\cdot,N)N$ everywhere. We thus get from Proposition \ref{allandrhoconstant}:

\begin{prop}\label{contactincsf}
Let $M$ be a contact hypersurface of an $n$-dimensional K\"{a}hler manifold $\bar{M}$ with constant holomorphic sectional curvature, $n \geq 3$. Then $M$ has constant mean curvature.
\end{prop}

For $n \geq 3$, the contact hypersurfaces in ${\mathbb C}^n$ and ${\mathbb C}P^n$ were classified by Okumura (\cite{O66}) and in ${\mathbb C}H^n$ by Vernon (\cite{V87}). 
A remarkable consequence of their classifications is the following:

\begin{cor}
Every complete contact hypersurface in a simply connected complete K\"{a}hler manifold with constant holomorphic sectional curvature is homogeneous..
\end{cor}

For $n = 2$, however, there are more contact hypersurfaces, as we will now show for ${\mathbb C}^2$.

\begin{thm}\label{classC2}
Let $C$ be a complex curve in ${\mathbb C}^2$ whose second fundamental form is nonzero at each point. Assume that $r \in {\mathbb R}_+$ is chosen so that $C$ has no focal point at distance $r$. Then the tube of radius $r$ around $C$ is a contact hypersurface of ${\mathbb C}^2$.
\end{thm}

\begin{proof}
Let $C$ be a complex curve in ${\mathbb C}^2$ whose second fundamental form is nonzero at each point. Assume that $r \in {\mathbb R}_+$ is chosen so that $C$ has no focal point at distance $r$. Then the tube $C_r$ of radius $r$ around $C$ is well-defined. Since every complex submanifold of a K\"{a}hler manifold is a minimal submanifold, the principal curvatures of $C$ with respect to a unit normal vector are of the form $\frac{1}{\theta}$ and $-\frac{1}{\theta}$ for some $\theta > 0$. The corresponding principal curvatures of the tube $C_r$ at the corresponding point are $\frac{1}{\theta - r}$ and $-\frac{1}{\theta + r}$ with the same principal curvature spaces via the usual identification of tangent vectors in a Euclidean space (see Theorem 8.2.2 in \cite{BCO03} for how to calculate the principal curvatures of tubes). These two principal curvature spaces span the maximal complex subspace ${\mathcal C}$ and hence we get $S \phi + \phi S = \frac{2r}{\theta^2 - r^2}\phi$, which shows that $C_r$ is a contact hypersurface. 
\end{proof}

\section{Contact hypersurfaces in  $Q^n$ and $Q^{n*}$}\label{application}

We now consider the case of the complex quadric $Q^n = SO_{n+2}/SO_nSO_2$ and its noncompact dual space $Q^{n*} = SO_{n,2}/SO_nSO_2$, $n \geq 3$. The complex quadric (and its noncompact dual) have two geometric structures which completely describe its Riemannian curvature tensor $\bar{R}$. The first geometric structure is of course the K\"{a}hler structure $(J,g)$. The second geometric structure is a rank two vector bundle ${\mathcal A}$ over $Q^n$ which contains an $S^1$-bundle of real structures on the tangent spaces of $Q^n$. 
This bundle has for instance been studied by Smyth in \cite{S67} in the context of complex hypersurfaces. The complex quadric $Q^n$ is a complex hypersurface in ${\mathbb C}P^{n+1}$ and the bundle ${\mathcal A}$ is just the family of shape operators with respect to the normal vectors in the rank two normal bundle. We refer also to \cite{BS13} for more details about ${\mathcal A}$.
The Riemannian curvature tensor $\bar{R}$ of $Q^n$ is given by
\begin{eqnarray*}
\bar{R}(X,Y)Z & = & g(Y,Z)X - g(X,Z)Y 
 +  g(JY,Z)JX - g(JX,Z)JY - 2g(JX,Y)JZ \\
 & & +\, g(AY,Z)AX - g(AX,Z)AY 
 + \,g(JAY,Z)JAX - g(JAX,Z)JAY,
\end{eqnarray*}
where $A$ is an arbitrary real structure in ${\mathcal A}$. 
For $Q^{n*}$ the Riemannian curvature tensor has the same form with a minus sign in front of it. For a real structure $A \in {\mathcal A}$ we denote by $V(A)$ its $(+1)$-eigenspace; then $JV(A)$ is the $(-1)$-eigenspace of $A$.
By ${\mathcal Q}$ we denote the maximal ${\mathcal A}$-invariant subbundle of $TM$.

A nonzero tangent vector $W$ of $Q^n$ resp.\ $Q^{n*}$ is called singular if it is tangent to more than one maximal flat in $Q^n$ resp.\ $Q^{n*}$. There are two types of singular tangent vectors in this situation:
\begin{itemize}
\item[(i)] If there exists a real structure $A \in {\mathcal A}$ such that $W \in V(A)$, then $W$ is singular. Such a singular tangent vector is called ${\mathcal A}$-principal.
\item[(ii)] If there exist a real structure $A \in {\mathcal A}$ and orthonormal vectors $X,Y \in V(A)$ such that $W/||W|| = (X+JY)/\sqrt{2}$, then $W$ is singular. Such a singular tangent vector is called ${\mathcal A}$-isotropic.
\end{itemize}
For every unit tangent vector $W$ of $Q^n$ resp.\ $Q^{n*}$ there exist a real structure $A \in {\mathcal A}$ and orthonormal vectors $X,Y \in V(A)$ such that
$W = \cos(t)X + \sin(t)JY$
for some $t \in [0,\pi/4]$. The singular tangent vectors correspond to the values $t = 0$ and $t = \pi/4$.

We now apply the results in Section \ref{general} to $Q^n$ (resp.\ $Q^{n*}$).
Inserting $X = JN$ and $Y = Z = N$ into the expression for the curvature tensor $\bar{R}$ of $Q^n$, and using the fact that $AJ = -JA$, we get
$ \bar{R}(JN,N)N =  4JN  + 2g(AN,N)AJN  - 2g(AJN,N)AN$.
If $N$ is ${\mathcal A}$-principal, that is, $AN = N$ for some real structure $A \in {\mathcal A}$, then we have 
$\bar{R}(JN,N)N =  2JN  $.
If $N$ is not ${\mathcal A}$-principal, then there exists a real structure $A \in {\mathcal A}$ such that
$N = \cos(t)Z_1 + \sin(t)JZ_2$
for some orthonormal vectors $Z_1,Z_2 \in V(A)$ and $0 < t \leq \frac{\pi}{4}$. This implies
$AN = \cos(t)Z_1 - \sin(t)JZ_2$, 
$JN =  \cos(t)JZ_1 - \sin(t)Z_2$ and 
$AJN = - \cos(t)JZ_1 - \sin(t)Z_2$.
Then we have $g(AN,N) = \cos(2t)$ and $g(AJN,N)  = 0$, and therefore
$\bar{R}(JN,N)N =  4JN  + 2\cos(2t)AJN$.
Thus $JN$ is an eigenvector of $\bar{R}_N$ if and only if $t = \frac{\pi}{4}$ or $AJN$ is a multiple of $JN$. Since both $\cos(t)$ and $\sin(t)$ are nonzero for $0 < t \leq \frac{\pi}{4}$ it is easy to see from the above expressions that $AJN$ is never a multiple of $JN$. Since $t = \frac{\pi}{4}$ if and only if $N$ is ${\mathcal A}$-isotropic we therefore conclude that $JN$ is an eigenvector of $\bar{R}_N$ everywhere if and only if $N$ is ${\mathcal A}$-principal or ${\mathcal A}$-isotropic everywhere. 
The Riemannian curvature tensor of the noncompact dual symmetric space $Q^{n*}$ is just the negative of the Riemannian curvature tensor of $Q^n$.
We therefore have proved:

\begin{prop}\label{JNquadric}
Let $M$ be a real hypersurface of the complex quadric $Q^n$ (resp.\ of its noncompact dual space $Q^{n*}$), $n \geq 3$. Then the following statements are equivalent:
\begin{itemize}
\item[(i)] $JN$ is an eigenvector of the normal Jacobi operator $\bar{R}_N = \bar{R}(\cdot,N)N$ everywhere;
\item[(ii)] $N$ is ${\mathcal A}$-principal or ${\mathcal A}$-isotropic everywhere;
\item[(iii)] $N$ is a singular tangent vector of $Q^n$ (resp.\ of $Q^{n*}$) everywhere.
\end{itemize}
\end{prop}

We now insert $X = JN$ and $Y = N$ into the equation for $\bar{R}$ and assume that $Z \in TM$. Then we get
$\bar{R}(JN,N)Z =  2\eta(Z)N + 2JZ  + 2g(AN,Z)AJN - 2g(AJN,Z)AN$.
In particular, for $Z \in {\mathcal C}$ this gives
$\bar{R}(JN,N)JZ =  -2Z  + 2g(AJN,Z)AJN + 2g(AN,Z)AN$.
If $Z \in {\mathcal C}$ is a principal curvature vector of $M$ with corresponding principal curvature $\lambda$, we obtain from Proposition \ref{Asquared}:

\begin{prop}\label{Asquaredquadric}
Let $M$ be a contact hypersurface of $Q^n$ resp.\ of $Q^{n*}$. Then we have
\begin{eqnarray*}
& & \epsilon (\lambda^2  - 2\rho \lambda + (\alpha\rho + \epsilon))Z \\
& = &  g(Z,AN)(AN-g(AN,N)N) +  g(Z,AJN)(AJN - g(AJN,JN)JN)  
\end{eqnarray*}
for all $Z\in {\mathcal C}$ with $SZ = \lambda Z$, where $\epsilon = +1$ for $Q^n$ and $\epsilon = -1$ for $Q^{n*}$.
\end{prop}

We will now investigate the normal vector field of a contact hypersurface.

\begin{prop}\label{normalquadric}
Let $M$ be a contact hypersurface of $Q^n$ resp.\ of $Q^{n*}$. Then the normal vector field $N$ cannot be ${\mathcal A}$-isotropic.
\end{prop}

\begin{proof}
We give the argument for $Q^n$, for $Q^{n*}$ it is analogous.
If $N$ is ${\mathcal A}$-isotropic we obtain from Proposition \ref{Asquaredquadric} that
$
(\lambda^2  - 2\rho \lambda + (\alpha\rho + 1))Z =   g(Z,AN)AN + g(Z,AJN)AJN = Z_{{\mathcal C} \ominus {\mathcal Q}}
$
for all $Z\in {\mathcal C}$ with $SZ = \lambda Z$. We decompose $Z = Z_{\mathcal Q} + Z_{{\mathcal C} \ominus {\mathcal Q}}$ into its ${\mathcal Q}$- and $({\mathcal C} \ominus {\mathcal Q})$-components. This implies
$(\lambda^2  - 2\rho \lambda + (\alpha\rho + 1))Z_{\mathcal Q} = 0$
and 
$(\lambda^2  - 2\rho \lambda +\alpha\rho)Z_{{\mathcal C} \ominus {\mathcal Q}} = 0$.
If $Z_{{\mathcal C} \ominus {\mathcal Q}} \neq 0$ then $\lambda^2  - 2\rho \lambda +\alpha\rho = 0$ and therefore $Z_{\mathcal Q} = 0$. It follows that ${\mathcal Q}$ and ${\mathcal C} \ominus {\mathcal Q}$ are invariant under the shape operator of $M$. 
There exists a one-form $q$ on $Q^n$ along $M$ such that $\bar\nabla_XA = q(X)JA$ for all $X \in TQ^n$ (see e.g.\ Proposition 7 in \cite{S67}). By differentiating the equation $g(AN,JN) = 0$ with respect to $X \in TM$ we get $g(SAJN,X) = 0$,
which implies $SAJN = 0$. By differentiating the equation $g(AN,N) = 0$ and using Proposition \ref{Okumura} we get
$0 =  g(SAN,X) $
for all $X \in TM$, which implies
$SAN = 0$ and thus $SAJN = 2\rho AJN$.
Altogether this yields  $\rho = 0$, which is a contradiction. It follows that $N$ cannot be ${\mathcal A}$-isotropic. 
\end{proof}

We now investigate the case when $N$ is ${\mathcal A}$-principal and $\bar{M} = Q^n$.
If $N$ is ${\mathcal A}$-principal, that is, if $AN = N$, we get from Proposition \ref{Asquaredquadric} that
$(\lambda^2  - 2\rho \lambda + (\alpha\rho + 1) )Z = 0$
for all $Z \in {\mathcal C}$ with $SZ = \lambda Z$.
Thus there are at most two distinct constant principal curvatures $\lambda$ and $\mu = 2\rho - \lambda$ on ${\mathcal C}$. 
We again use the fact that there exists a one-form $q$ on $Q^n$ along $M$ such that $\bar\nabla_XA = q(X)JA$ for all $X \in TQ^n$. By differentiating the equation $g(AN,JN) = 0$ with respect to $X \in TM$ we get 
$q(X) = 2g(ASX,JN) =  2g(SAJN,X) = - 2g(SJAN,X) = - 2g(SJN,X) 
= -2\alpha g(JN,X) =  2\alpha \eta(X)$.
It follows that $\bar\nabla_XA = 0$ for all $X \in {\mathcal C}$. From $AN = N$ we get $AJN = -JAN = -JN$. Differentiating this equation with respect to $X \in {\mathcal C}$ gives $ASX = SX$. Thus, for all $Z \in {\mathcal C}$ with $SZ = \lambda Z$ resp.\ $SZ = \mu Z$ we get $\lambda AZ = \lambda Z$ and $\mu AZ = \mu Z$. If both $\lambda$ and $\mu$ are nonzero this implies $AZ = Z$ for all $Z \in {\mathcal C}$ and hence ${\rm tr}(A) = 2(n-1)$, which contradicts the fact that $A$ is a real structure and hence ${\rm tr}(A) = 0$. We thus may assume that $\lambda = 0$. If the corresponding principal curvature space $T_\lambda$ is $J$-invariant this implies $\rho = 0$, which is a contradiction. We thus must have $0 \neq \mu = 2\rho$ and $JT_\lambda = T_\mu$.
Thus we have shown that there are exactly two distinct constant principal curvatures $\lambda = 0$ and $\mu = 2\rho$ on ${\mathcal C}$. Moreover, we have
$J T_\lambda = T_\mu$ for the corresponding principal curvature spaces $T_\lambda$ and $T_\mu$. 
From Proposition \ref{Asquaredquadric} we also get the equation
 $\alpha \rho + 1 = 0$. 

For the dual manifold $Q^{n*}$ we have to consider the equation 
$(\lambda^2  - 2\rho \lambda + (\alpha\rho - 1) )Z = 0$,
but all other arguments remain the same.
Thus we have proved:

\begin{prop}\label{Aprincipalquadric}
Let $M$ be a contact hypersurface of $Q^n$ resp.\ of $Q^{n*}$ and assume that the normal vector field $N$ is ${\mathcal A}$-principal. Then $M$ has three distinct constant principal curvature $\alpha$, $\lambda = 0$ and $\mu = 2\rho$ with corresponding principal curvature spaces ${\mathbb R}JN$, $T_\lambda \subset {\mathcal C}$ and $T_\mu \subset {\mathcal C}$ satisfying $JT_\lambda = T_\mu$. The principal curvature $\alpha$ is given by $\alpha = -\frac{1}{\rho}$ when $\bar{M} = Q^n$ and $\alpha = \frac{1}{\rho}$ when $\bar{M} = Q^{n*}$. Moreover, $T_\mu = V(A) \cap {\mathcal C}$ and $T_\lambda = JV(A) \cap{\mathcal C}$.
\end{prop}

We now prove Theorem \ref{classquadric}. 
Without loss of generality we may assume that $\rho > 0$ (otherwise replace $N$ by $-N$). Since $\rho$ is constant there exists $0 < r < \frac{\pi}{2\sqrt{2}} $ such that $\rho = \frac{1}{\sqrt{2}} \tan(\sqrt{2}r)$. From Proposition \ref{Aprincipalquadric} we then get $\alpha = -\sqrt{2}\cot(\sqrt{2}r)$ and $\mu = \sqrt{2}\tan(\sqrt{2}r)$. 

The normal Jacobi operator $\bar{R}_N = \bar{R}(\cdot,N)N$ has two eigenvalues $0$ and $2$ with corresponding eigenspaces $T_\lambda \oplus {\mathbb R}N$ and $T_\mu \oplus {\mathbb R}JN$. Thus the normal Jacobi operator and the shape operator of $M$ commute, which allows us to use Jacobi field theory to determine explicitly the focal points of $M$. For $p \in M$ we denote by $\gamma_p$ the geodesic in $Q^n$ with $\gamma_p(0)=p$ and $\dot{\gamma}_p(0)=N_p$, and by $F$ the smooth map
$F: M \to Q^n\ ,\ p \mapsto \gamma_p(r)$.
The differential $d_pF$ of $F$ at $p$ can be computed using Jacobi vector fields by means of
$d_pF(X)=Y_X(r)$,
where $Y_X$ is the Jacobi vector field along $\gamma_p$ with initial values $Y_X(0)=X$ and $Y_X'(0)=-SX$.  
This leads to the following expressions for the Jacobi vector fields along ${\gamma}_p$:
\begin{equation*}
Y_X(r)=\begin{cases}
(\cos(\sqrt{2}r) - \frac{\alpha}{\sqrt{2}}\sin(\sqrt{2}r))E_X(r) & \text{if $X \in T_\alpha$,}\\
(\cos(\sqrt{2}r) - \frac{\mu}{\sqrt{2}}\sin(\sqrt{2}r))E_X(r) & \text{if $X\in T_\mu$,}\\
E_X(r) & \text{if $X\in T_\lambda$},
\end{cases}
\end{equation*}
where $E_X$ is the parallel vector field along ${\gamma}_p$ with $E_X(0)=X$.
This shows that ${\rm Ker}(dF) = T_\mu$ and thus $F$ has constant rank $n$.
Therefore, locally, $F$ is a submersion onto a submanifold $P$ of $Q^n$ of real dimension $n$. Moreover, the tangent space $T_{F(p)}P$ of $P$ at $F(p)$ is obtained by parallel translation of $(T_\lambda \oplus T_\alpha)(p)$ along $\gamma_{p}$. Thus the submanifold $P$ is a totally real submanifold of $Q^n$ of real dimension $n$.

The vector $\eta_{p} = \dot{\gamma}_{p}(r)$ is a unit normal vector of $P$
at $F(p)$ and the shape operator $S_{\eta_{p}}$ of $P$ with respect to
$\eta_{p}$ can be calculated from the equation
$S_{\eta_{p}}Y_X(r) = - Y_X^\prime(r)$,
where $X \in (T_\lambda \oplus T_\alpha)(p)$. The above expression for the Jacobi vector fields $Y_X$ implies $Y_X^\prime(r) = 0$ for $X \in T_\lambda(p)$ and $X \in T_{\alpha}(p)$, and therefore $S_{\eta_{p}} = 0$.
The vectors of the form $\eta_{q}$, $q \in F^{-1}(\{F(p)\})$,
form an open subset of the unit
sphere in the normal space of $P$ at $F(p)$. Since $S_{\eta_{q}}$ vanishes
for all $\eta_{q}$ it follows that $P$ is an $n$-dimensional totally geodesic totally real submanifold of $Q^n$.
Rigidity of totally geodesic submanifolds now
implies that the entire submanifold
$M$ is an open part of a tube of radius $r$
around an $n$-dimensional connected, complete,
totally geodesic, totally real submanifold $P$ of $Q^n$. Such a submanifold is also known as a real form of $Q^n$. The real forms of the complex quadric $Q^n$ are well-known, see for example \cite{CN77} or \cite{K08}. This shows that $P$ is either a sphere $S^n$ or $(S^a \times S^b)/\{{\pm I}\}$ with $a + b = n$ and $a,b \geq 1$. However, we see from the above calculations that at each point the tangent space of $P$ corresponds to the $(-1)$-eigenspace of a real structure on $Q^n$, which rules out $(S^a \times S^b)/\{{\pm I}\}$. It follows that $P$ is a sphere $S^n$ embedded in $Q^n$ as a real form.

We remark that the focal set of a real form $S^n$ in $Q^n$ is a totally geodesic complex hyperquadric $Q^{n-1} \subset Q^n$. So the tubes around $S^n$ can also be regarded as tubes around $Q^{n-1}$. 

We now proceed with the proof of Theorem \ref{classquadricdual}.
We again assume $\rho > 0$. We can write  $\rho = \frac{1}{\sqrt{2}} \tanh(\sqrt{2}r)$, $\rho = \frac{1}{\sqrt{2}} $ or $\rho = \frac{1}{\sqrt{2}} \coth(\sqrt{2}r)$ with some $r \in {\mathbb R}_+$.
The normal Jacobi operator $\bar{R}_N = \bar{R}(\cdot,N)N$ has two eigenvalues $0$ and $-2$ with corresponding eigenspaces $T_\lambda \oplus {\mathbb R}N$ and $T_\mu \oplus {\mathbb R}JN$.
The Jacobi vector fields are given by 
\begin{equation*}
Y_X(r)=\begin{cases}
(\cosh(\sqrt{2}r) - \frac{\alpha}{\sqrt{2}}\sinh(\sqrt{2}r))E_X(r) & \text{if $X \in T_\alpha$,}\\
(\cosh(\sqrt{2}r) - \frac{\mu}{\sqrt{2}}\sinh(\sqrt{2}r))E_X(r) & \text{if $X\in T_\mu$,}\\
E_X(r) & \text{if $X\in T_\lambda$},
\end{cases}
\end{equation*} 
We now distinguish  three cases:

Case 1: $\rho = \frac{1}{\sqrt{2}} \tanh(\sqrt{2}r)$. From Proposition \ref{Aprincipalquadric} we then get $\alpha = \sqrt{2}\coth(\sqrt{2}r)$ and $\mu = \sqrt{2}\tanh(\sqrt{2}r)$. 
Here we get ${\rm Ker}(dF) = T_\alpha$ and thus $F$ has constant rank $2(n-1)$. Analogously to the compact case we can deduce that $M$ is an open part of a tube around a totally geodesic, complex submanifold $P$ of $Q^{n*}$ of complex dimension $n-1$. Using duality of symmetric spaces and \cite{K08} we can see that $P$ is a totally geodesic $Q^{(n-1)*}$ in $Q^{n*}$.

Case 2: $\rho = \frac{1}{\sqrt{2}}$. From Proposition \ref{Aprincipalquadric} we then get $\alpha = \mu = \sqrt{2}$. In this case $M$ does not have any focal points. For $X \in T_\alpha$ or $X \in T_\mu$ the Jacobi vector fields are $Y_X(r) = \exp(-\sqrt{2}r)E_X(r)$ and remain bounded for $r \to \infty$. Together with $Y_X(r) = E_X(r)$ for $X \in T_\lambda$ this implies that all normal geodesics of $M$ are asymptotic to each other. From this we see that $M$ is a horosphere whose center at infinity is given by an equivalence class of asymptotic geodesics whose tangent vectors are all ${\mathcal A}$-principal. Thus the center at infinity of the horosphere is a singular point of type ${\mathcal A}$-principal.

Case 3: $\rho = \frac{1}{\sqrt{2}} \coth(\sqrt{2}r)$. From Proposition \ref{Aprincipalquadric} we then get $\alpha = \sqrt{2}\tanh(\sqrt{2}r)$ and $\mu = \sqrt{2}\coth(\sqrt{2}r)$. In this situation we get ${\rm Ker}(dF) = T_\mu$ and thus $F$ has constant rank $n$. This case is analogous to the compact situation and we deduce that $M$ is an open part of a tube around a real form ${\mathbb R}H^n$ of $Q^{n*}$. 

The arguments given above can be carried out in reverse order to show that all the resulting hypersurfaces are in fact contact hypersurfaces.
This finishes the proof of Theorems \ref{classquadric} and  \ref{classquadricdual}

\end{document}